\documentclass[11pt]{amsart}
\usepackage{palatino}

\usepackage{amsmath,amsthm}     
\usepackage{graphicx}     

\usepackage{url}
\usepackage{amsfonts} 
\usepackage[marginratio=1:1,height=8.5in,width=6.0in]{geometry}

\newtheorem{theorem}{Theorem}

\theoremstyle{definition}

\newtheorem*{remark}{Remark}

\newcommand{\E}{\mathrm{E}}
\newcommand{\var}{\mathrm{Var}}

\newcommand{\pn}{\mathcal{P}}
\newcommand{\cov}{\mathrm{Cov}}
\newcommand{\pr}{\mathrm{P}}
\newcommand{\abs}[1]{\left\lvert#1\right\rvert}
\newcommand{\ra}{\rightarrow}

\newtheorem*{question*}{Question}
\newtheorem{corollary}[theorem]{Corollary} 
\newtheorem{example}[theorem]{Example}
\newtheorem{lemma}[theorem]{Lemma}  

\allowdisplaybreaks

\makeatletter
\@addtoreset{footnote}{page}
\makeatother

\def\abrule{\vrule height11.5pt width0pt depth5.5pt}

\def\brule{\vrule height0pt width0pt depth5.5pt}

\begin{document}

\title{Matching Adjacent Cards}

\author{Kent E. Morrison}         
\address{American Institute of Mathematics,
Caltech, Pasadena, CA 91125}

\thanks{\emph{Mathematics Magazine} \textbf{97} (2024) 471--483.}

\begin{abstract}
In a well-shuffled deck of cards, what is the probability that somewhere in the 
deck there are adjacent cards of the same rank? What is the average number 
of adjacent matches? What is the probability distribution for the number of matches? 
We answer these and related questions for both the standard $52$-card deck 
with four suits and $13$ ranks and for generalized decks with $k$ suits and $n$ ranks. 
We also determine the limiting distribution as $n$ goes to infinity with $k$ fixed.
\end{abstract}
\maketitle

\large
\renewcommand{\baselinestretch}{1.2}   
\normalsize

Here is my offer. First pay a dollar to play the game. Then take a shuffled deck of 
cards and fan out the entire deck face up onto the table. For each occurrence 
of adjacent cards with the same value I will pay you a dollar. If there is one 
match, then you break even. If there are no matches, then you lose your dollar. 
If there are two or more matches, then you come out ahead. Should you accept my offer?  

Get a deck of cards and give it a try. Play at least five times, making sure 
to shuffle the deck several times before each round.  

Playing a few rounds should convince you to take the bet. 
But would you be willing to pay two dollars to play? What is the maximum 
you should consider paying? And what is the probability of losing your dollar? 
We get a better idea with some computer simulation of a hundred thousand games. 

In the appendix, you can find code in Mathematica and in Sage to do this. 
Note that three cards of the same rank occurring consecutively count as two matches, 
and all four cards of one rank occurring consecutively count as three matches.  
The maximal number of matches possible is $39$, the probability of which is 
exceedingly small---approximately $6.76672 \cdot 10^{-41}$.  
(Exercise: find the exact answer.) 

For one particular run with $100$,$000$ shuffles, the match count histogram 
is shown in Figure~\ref{fig1}. Although you cannot see this in the histogram, 
the maximal count that occurred is $13$, and that happened just once. 
The count of $11$ occurred $12$ times and the count of $12$ occurred twice. 
For this run the average was $3.00125$, and so you might be willing to pay up 
to about three dollars per game. Also, the probability of no matches appears to 
be between $4\%$ and $5\%$. The most frequent match counts are 
$2$ and $3$, each occurring more than $20\%$ of the time. 
\begin{figure}[h]
\begin{center}
\includegraphics[width=3in]{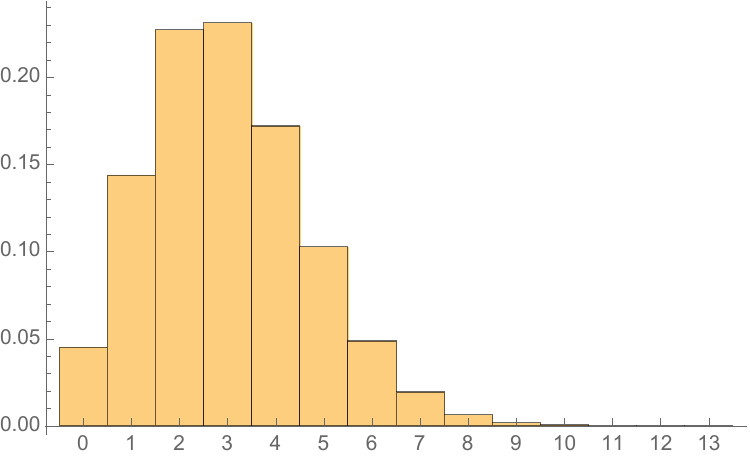}
\caption{A histogram displaying the results of playing our card-matching
game $100$,$000$ times.}
\label{fig1}
\end{center}
\end{figure}
It is fairly easy to find the expected number of matches using indicator random variables. 
Let $M_i$ be $1$ if card $i$ and card $i+1$ match and $0$ otherwise. 
Then the number of matches $M$ is the sum $M_1 +M_2 +\cdots + M_{51}$. 
Expectation is linear, so $\E(M)= \E(M_1) + \E(M_2)+\cdots + \E(M_{51})$, 
but $\E(M_i)$ is the same for all $i$. Finally, $E(M_i)=3/51$ because there are 
$51$ possibilities for card $i+1$, and three of them match card $i$. 
Therefore, $\E(M)=3$. This easily generalizes:
\begin{theorem}
In a randomly arranged deck of $kn$ cards consisting of $k$ 
suits with $n$ cards in each suit, the expected number of matches is $k-1$.
\end{theorem}

\begin{proof}
For $1 \le i \le kn-1$, let $M_i$ be the indicator random variable for the event 
that cards $i$ and $i+1$ match, and let $M=\sum M_i$. Then $\E(M_i)= (k-1)/(kn-1)$ 
because there are $k-1$ cards of the same rank as card $i$ and $kn-1$ 
possibilities for the next card. Therefore,
\[ 
\E(M)=\sum_{i=1}^{kn-1}  \E(M_i) = \sum_{i=1}^{kn-1} \frac{k-1}{kn-1} = k-1.
\]
\end{proof}

Having found the expected number of matches, we move on to the general 
question of the complete probability distribution for $M$. With a standard 
$52$-card deck the number of matches is at most $39$ since each rank 
contributes at most three matches. For general $k$ and $n$, 
the number of matches is at most $kn-n$.
\begin{question*}
For a well-shuffled deck consisting of $k$ suits and $n$ ranks, 
what is the probability that there are exactly $r$ matches for $r=0,1,2,\ldots,kn-n$?
\end{question*}

We need some basic definitions. First, we define a card deck with $k$ suits and 
$n$ ranks to be the multiset of size $kn$ made up of $k$ copies of each  
integer from $1$ to $n$. For example, for $k=3$ and $n=4$ the deck is 
$\{1,1,1,2,2,2,3,3,3,4,4,4\}$. Now we no longer see different suits; 
the aces are all identical. The number of distinct permutations is no longer 
$(kn)!$ because the cards in each of the $n$ ranks can be permuted among 
themselves in $k!$ ways without changing what we see. Therefore, the number 
of permutations is $(kn)!/(k!)^n$. Letting $\alpha_r$ be the number of permutations 
with exactly $r$ matches we have
\[ 
\pr(M=r)= \frac{\alpha_r}{(kn)!/(k!)^n}=\frac{(k!)^n}{(kn)!}\alpha_r.
\]
We now focus on the problem of counting the number of permutations with 
$r$ matches for general $k$ and $n$. We will deal with $k=4$ in detail before 
tackling arbitrary values of $k$.

Counting permutations with no matches (i.e., $r=0$) has been of particular interest. 
Eriksson and Martin~\cite{Eriksson&Martin2017} named them ``Carlitz permutations'' 
and have enumerated them for $k=2,3,4$. The case with $k=2$ is similar to a linear 
variation of the relaxed version of the ``Probl\`eme des m\'enages.'' The classic 
version asks for the number of ways that $n$ couples can be seated at a circular 
table so that the men and women alternate and no one sits next to his or her partner. 
The relaxed version~\cite{Bogart&Doyle86} drops the requirement that the 
genders alternate. Now suppose that they are to be arranged in a single line rather 
than in a circle. Then the number of possible linear arrangements is 
$2^n$ times $\alpha_0$, the number of Carlitz permutations with two suits.  
(See sequences A114938 and A007060 in the  On-Line Encyclopedia of Integer 
Sequences (OEIS) at \url{https://oeis.org}.)

\section{Four suits ($k=4$)}

The deck is the multiset of size $4n$ consisting of four copies of each number 
$1$ through $n$, and the number of permutations is $(4n)!/(4!)^n$. 
Let $\alpha_r$ be the number of permutations with exactly $r$ matches and let  
\[ 
A(x)=\sum_{r=0}^{3n} \alpha_r x^r 
\] 
be the associated generating function. In order to use the principle of 
inclusion-exclusion we define another generating function $B(x)=\sum \beta_m x^m$, 
which is related to $A(x)$ by the functional equation $A(x)=B(x-1)$ and 
with the virtue that the $\beta_m$ are easier to calculate than the $\alpha_r$. 
Essentially, $\beta_m$ counts permutations that have at least $m$ matches.

One rank with four cards $x,x,x,x$ will account for up to three matches. 
We can be sure that this rank contributes at least one match by glueing two of the cards
together: $xx,x,x$. To force at least two matches, we can either glue two 
pairs together as $xx,xx$ or glue three cards together as $xxx,x$. 
To force three matches we glue all four cards together: $xxxx$. 
Now we choose a pattern for each of the $n$ ranks.  
Let $s,t,u,v,w$ denote the number of ranks having each of the five possible patterns, 
according to this scheme:
\begin{equation*}
\begin{array}{cc}
s &  x \cdot x \cdot x \cdot x \\t & xx \cdot x \cdot x  \\
u & xxx \cdot x \\v & xx \cdot xx \\w & xxxx
\end{array}
\end{equation*}
For a given choice of $s,t,u,v,w$ (such that $s+t+u+v+w=n$) the number of objects 
being permuted is no longer $4n$, but rather $4s+3t+2u+2v+w$. 
The pattern $ x \cdot x \cdot x \cdot x$ has four identical objects, 
and the patterns $xx \cdot x \cdot x $ and $ xx \cdot xx$ each have two identical objects, 
and so we divide by the redundancy factor $(4!)^s (2!)^t (2!)^v$.  
This construction gives a list of
\begin{equation}
\frac{(4s+3t + 2u+2v+w)!}{(4!)^s (2!)^t (2!)^v}  \label{number-multiperms}
\end{equation}
permutations. A permutation with these parameters has at least 
$t+2u+2v+3w$ matches. Define $\beta_m$ for $m=0,1,\ldots,3n$ 
to be the sum of the cardinalities of all the lists such that $m=t+2u+2v+3w$. Thus,
\begin{align}  \label{beta_m}
\beta_m &= \sum_{\substack{s+t+u+v+w=n \\ 
t+2u+2v+3w=m}} \binom{n}{s,t,u,v,w}
\frac{(4s+3t + 2u+2v+w)!}{(4!)^s (2!)^t (2!)^v}   \nonumber \\ 
&=\sum_{\substack{s+t+u+v+w=n \\ t+2u+2v+3w=m}} 
\binom{n}{s,t,u,v,w}\frac{(4n-m)!}{(4!)^s (2!)^t (2!)^v}. 
\end{align}

The permutations with exactly $m$ matches occur once in all of the permutations 
counted by $\beta_m$, but the permutations with more than $m$ matches 
occur many times. For $r > m$, a permutation with $r$ matches is 
generated $\binom{r}{m}$ times because each choice of $m$ of its $r$ matches 
corresponds to a different set of parameters $s,t,u,v,w$.  Therefore,
\[ 
\beta_m = \sum_{r \ge m} \binom{r}{m} \alpha_r,
\]
and 
\begin{align*}
B(x) &= \sum_{m=0}^{3n} \beta_m x^m 
= \sum_{m=0}^{3n}  \sum_{r=m}^{3n} \binom{r}{m} \alpha_r x^m \\
&= \sum_{r=0}^{3n} \alpha_r \sum_{m=0}^r \binom{r}{m} x^m 
= \sum_{r=0}^{3n} \alpha_r (x+1)^r = A(x+1).
\end{align*}
Equivalently, $A(x) = B(x-1)$,
which is the principle of inclusion-exclusion, concisely stated. Expanding
\[ 
B(x-1) = \sum_{m = 0}^{3n} \beta_m (x-1)^m 
\]
and equating coefficients gives
\begin{equation*} 
\alpha_r = \sum_{m=r}^{3n} (-1)^{m-r} \binom{m}{r} \beta_m .
\end{equation*} 
Using equation~\eqref{beta_m}, we get
\begin{equation*} 
\alpha_r =\sum_{m=r}^{3n} (-1)^{m-r} 
\binom{m}{r} \sum_{\substack{s+t+u+v+w=n \\ t+2u+2v+3w=m}} 
\binom{n}{s,t,u,v,w}\frac{(4n-m)!}{(4!)^s (2!)^t (2!)^v} .     
\end{equation*}
Since $m=t+2u+2v+3w$, and  
\[
(-1)^{m-r} = (-1)^{t+2u+2v+3w - r} = (-1)^{t + w -r },
\] 
we can eliminate $m$ from the formula for $\alpha_r$ and write it as a single sum.
\begin{theorem} \label{four-suit-matches} 
\begin{multline*}
\alpha_r =\sum_{\substack{s+t+u+v+w=n \\ t+2u+2v+3w \geq r}}  
(-1)^{t+w-r} \binom{t+2u+2v+3w}{r} \\ \times \binom{n}{s,t,u,v,w}
\frac{(4s+3t+2u+2v+w)!}{(4!)^s(2!)^t(2!)^v}.
\end{multline*}
\end{theorem}

\begin{corollary} \label{alpha_0}
The number of permutations with no matches is 
\begin{equation*}  
\alpha_0 =  \sum_{s+t+u+v+w=n } (-1)^{t+w} \binom{n}{s,t,u,v,w} \frac{(4s+3t+2u+2v+w)!}{(4!)^s(2!)^t(2!)^v }.
\end{equation*}
\end{corollary}

\begin{remark}
The values of $\alpha_0$ are given by the sequence A321633 in the OEIS. 
For $0 \le n \le 5$ the terms are 
\[
1, \quad 0, \quad 2, \quad 1092, \quad 2265024, \quad 11804626080.
\]
\end{remark}

\begin{corollary} \label{pr_M=r}
For a deck with four suits and $n$ cards in each suit, we have
\begin{equation*}
\pr(M=r)= \frac{\alpha_r}{(4n)!/(24)^n}, \quad  r=0,1,\ldots,3n.
\end{equation*} 
\end{corollary}

\begin{example} \label{exm:k=4,n=2}
For $k=4$ and $n=2$, there are $70$ permutations of $\{1,1,1,1,2,2,2,2\}$. 
Thus, $\beta_0=70$. It is easy to see that $\beta_6=2$. We will find $\beta_3$ 
and leave the rest as an exercise. To force at least three matches, the two ranks, 
denote them $x$ and $y$, fit one of three patterns
\begin{align*}    
xxxx \cdot y \cdot y \cdot y \cdot y    &\quad\quad   s=1, t=u=v=0, w=1 ,  \\
xxx \cdot x \cdot yy \cdot y \cdot y    &\quad\quad   s=0, t=1, u=1, v=0, w=0 ,\\
xx \cdot xx \cdot yy \cdot y \cdot y    &\quad\quad   s=0, t=1, u=0, v=1, w=0 .
\end{align*} 
For each pattern, the multinomial coefficient $\binom{n}{s,t,u,v,w}=2$, 
and so
\[ 
\beta_3 = 2\left(\frac{5!}{4!} +\frac{5!}{2!}+\frac{5!}{2! \,2!}\right) = 190.
\]  
After calculating the remaining $\beta_m$, you will get 
\[ 
B(x)= 70 + 210x + 270x^2 +190x^3 +  78x^4 +18x^5 + 2x^6 . 
\]     
Then
\begin{align*} 
A(x) &= 70 + 210(x-1)+ 270(x-1)^2 +190(x-1)^3 \\ 
&\qquad +  78(x-1)^4 +18(x-1)^5 + 2(x-1)^6 \\
&= 2 + 6x + 18x^2 +18x^3 + 18x^4 + 6x^5 + 2x^6.
\end{align*}  
An easy partial check is to list the six permutations with exactly one match.         
\end{example}

\section{The general case ($k$ suits)}

The patterns in which the $k$ cards of a given rank can be grouped correspond 
to partitions of $k$. As we have seen, there are five partitions for $k=4$. 
We represent a partition of $k$ by a vector $\pi=(\pi_1,\pi_2,\ldots,\pi_k)$, 
where $\pi_i$ is the number of occurrences of $i$. Therefore, $\sum_i i\pi_i =k$ 
and $0 \le \pi_i \le k$. Let $\nu(\pi) = \pi_1 + \cdots + \pi_k$ denote the number 
of parts in $\pi$, and let $\pn(k)$ be the set of partitions of $k$. 

Suppose that for each rank $1$ to $n$, we choose a partition describing 
the grouping of the cards in that rank. Let $s_\pi$ be the number of ranks 
which follow the partition $\pi$. Thus, $\sum_{\pi \in \pn(k)} s_\pi = n$. 
Treating each group as a single object means that for a rank whose cards are 
grouped according to $\pi$, there are $\nu(\pi)$ objects to be permuted. The total 
number of objects to be permuted is $\sum_\pi s_\pi \nu(\pi)$, and so the factorial 
of this is the total number of permutations. But we need to divide by factors to account 
for identical symbols, i.e., for card groups of the same size in a rank. 
Thus, for each partition $\pi$ we have a factor $\pi_i !$ repeated $s_\pi$ times.
 Therefore, with this choice of a partition for each rank, we generate a list of 
permutations of size
\[
\frac{\left(\sum\limits_{\pi \in \pn(k))}s_\pi \nu(\pi)\right)!}
{\prod_{\pi \in \pn(k)} \prod_{i=1}^k (\pi_i !)^{s_\pi}}.
\]

It may help to examine the situation with $k=4$ in order to see that this 
formula agrees with equation~(\ref{number-multiperms}). There the numbers 
$s,t,u,v,w$ are the values $s_\pi$ for 
$\pi=(4,0,0,0), (2,1,0,0)$, $(1,0,1,0),(0,2,0,0),(0,0,0,1)$.
\begin{equation*}
\begin{array}{cccc}s_\pi & \pi & \pi& \nu(\pi) \brule \\
\hline 
s &  x \cdot x \cdot x \cdot x & (4,0,0,0) & 4 \abrule\\
t & xx \cdot x \cdot x & (2,1,0,0) & 3 \brule \\
u & xxx \cdot x    &  (1,0,1,0)  &  2 \brule\\
v & xx \cdot xx & (0,2,0,0)  & 2 \brule\\
w & xxxx    & (0,0,0,1)  & 1 \\
\end{array}
\end{equation*}
The numerator is $(4s+3t+2u+2v+w)!$. The denominator is the product of 
powers of the factorials of all the entries of all the $\pi$. Many of them are 
powers of $0!$ and $1!$, which we do not write, and so we get only three 
interesting factors:  $(4!)^s (2!)^t (2!)^v$.

To define $\beta_m$, we construct permutations that have at least $m$ matches. 
Grouping the cards in a single rank according to the partition $\pi$ guarantees at 
least $\sum_i \pi_i(i~-~1)$ matches because each part of size $i$ contributes $i-1$ matches.
Let 
\[
\mu(\pi)=\sum_i \pi_i (i-1).
\] 
The number of matches for a permutation 
with type $(s_\pi)$ is at least $\sum_\pi s_\pi \mu(\pi)$. The number of ways to 
assign the patterns to suits to in order to have type $(s_\pi)$, 
where $\sum_\pi s_\pi = n$, is the multinomial coefficient $\binom{n}{(s_\pi)}$. 
Define $\beta_m$ by
\begin{equation*}
\beta_m = \sum_{ \substack{ \sum_\pi s_\pi =n\\
\sum_\pi s_\pi \mu(\pi)=m   } }
\binom{n}{(s_\pi)} \frac{(\sum_{\pi \in \pn(k)}
s_\pi \nu(\pi))!}{\prod_{\pi \in \pn(k)} \prod_{i=1}^k (\pi_i !)^{s_\pi}}.
\end{equation*}

Immediately from the definition of $\nu(\pi)$ and $\mu(\pi)$, 
we see that $\nu(\pi) + \mu(\pi) =n$, and so
\[
\sum_{\pi \in \pn(k)}s_\pi \nu(\pi) = \sum_{\pi \in \pn(k)}s_\pi (n-\mu(\pi)) 
= \sum_{\pi \in \pn(k)}s_\pi n -  \sum_{\pi \in \pn(k)}s_\pi \mu(\pi) =kn-m.
\]
Therefore,
\begin{equation*}
\beta_m = \sum_{ \substack{ \sum_\pi s_\pi =n\\
\sum_\pi s_\pi \mu(\pi)=m   } }
\binom{n}{(s_\pi)} \frac{(kn-m)!}{\prod_{\pi \in \pn(k)} 
\prod_{i=1}^k (\pi_i !)^{s_\pi}}.
\end{equation*}
Let 
\[
B(x) = \sum_{m=0}^{(k-1)n }\beta_m x^m
\quad \textrm{and} \quad 
A(x)= \sum_{r=0}^{(k-1)n } \alpha_r x^r, 
\]
where $\alpha_r$ is the number of permutations with exactly $r$ matches. 
Then $A(x)=B(x~-~1)$, and so
\begin{equation*}
\alpha_r = \sum_{m=r}^{(k-1)n} (-1)^{m-r} \binom{m}{r}\beta_m.
\end{equation*}

\section{Binomial and Poisson approximation}

The random variable $M$ is the sum of identically distributed Bernoulli random variables, 
but they are not independent. The dependence, however, is rather weak: whether 
or not cards $1$ and $2$ match does not have much influence on whether or not cards 
$9$ and $10$ match. So, we consider a related random variable, call it $M'$, 
which is the sum of $kn-1$ \emph{independent} Bernoulli random variables with 
success probability $p=(k-1)/(kn-1)$. Thus, $M$ and $M'$ have the same expected value, 
and $M'$ has a binomial distribution, which is much easier to compute:
\[   
\pr(M'=r) = \binom{kn-1}{r} p^r (1-p)^{kn-1-r}.
\]
As we will see later (Corollary~\ref{same-variance}) it turns out 
that---somewhat surprisingly---they also have the same variance. 
Can we use $M'$ as a good approximation to $M$?

There is another approximation to consider. When $n$ is large, the probability 
$p$ of a match at any single site is small, but the expected number of matches 
is $k-1$, independent of $n$. This suggests that a Poisson random variable with 
parameter $\lambda = k-1$ may also be a good approximation. 
Recall that a Poisson random variable $X$ with parameter $\lambda > 0$ 
has a distribution given by 
\[ 
\pr(X=k) = e^{-\lambda} \frac{\lambda^k}{k!}  \,\text{  for $k=0,1,2,\ldots$.} 
\]

Table~\ref{table1} shows the computer calculations for the distributions of 
$M$, $M'$, and $X$ for the standard deck where $n=13$ and $k=4$. 
For $M$, we use the formulas from Theorem~\ref{four-suit-matches} 
and Corollary~\ref{pr_M=r}. See the appendix for the code in  
Mathematica and Sage. The probabilities are rounded to the fifth decimal place. 
For $r \ge 14$ the probabilities round to $0$.
\begin{table}[htp]
\caption{$M$ for $k=4$,  $n=13$, $M' \sim \text{Binomial}(51,1/17)$, 
$X \sim \text{Poisson}(3)$.}
\begin{center}
\begin{tabular}{|cccc|}
\hline
$r$ & $\pr(M=r)$ & $\pr(M'=r)$ & $\pr(X=r)$ \abrule\\
\hline
 0 & 0.04548 & 0.04542 & 0.04979 \abrule\\
 1 & 0.14477 & 0.14477 & 0.14936 \brule\\
 2 & 0.22611 & 0.22620 & 0.22404\brule\\
 3 & 0.23085 & 0.23091 & 0.22404\brule\\
 4 & 0.17321 & 0.17319 & 0.16803\brule\\
 5 & 0.10181 & 0.10175 & 0.10082\brule\\
 6 & 0.04879 & 0.04875 & 0.05041\brule\\
 7 & 0.01959 & 0.01959 & 0.0216\brule\\
 8 & 0.00672 & 0.00673 & 0.0081\brule\\
 9 & 0.00200 & 0.00201 & 0.0027\brule\\
 10 & 0.00052 & 0.00053 & 0.00081\brule\\
 11 & 0.00012 & 0.00012 & 0.00022\brule\\
 12 & 0.00002 & 0.00003 & 0.00006\brule\\
 13 & 0.00000 & 0.00000 & 0.00001\brule\\
 \hline
\end{tabular}
\end{center}
\label{table1}
\end{table}%

For all practical purposes, the distributions of $M$ and $M'$ are the same. 
To measure how far apart they are, we use the \emph{total variation distance}
\begin{align*}
d_{TV}(M, M') &:= \frac{1}{2} \sum_{r=0}^{51} 
\abs{\pr(M=r)-\pr(M'=r)} \\ &\,= 0.000181682. 
\end{align*}                         
An equivalent formulation of the total variation distance is 
\[  
d_{TV}(M, M') = \sup_E |\pr(M \in E)-\pr(M' \in E)|, 
\]
where $E$ ranges over all events, which in this case are the subsets 
of $\{0,1,2,\ldots,51\}$. Thus, the maximal difference between the probabilities 
on any event is less than two hundredths of one percent. 

As $n$ increases, the dependence among the $M_i$ decreases, and so we expect 
that $M$ and $M'$ become ever closer. The limit distribution of $M'$ is Poisson 
with $\lambda=3$, and so the same should hold for $M$. We can prove that by 
showing that the total variation distance between the distributions of $M$ and $M'$ 
goes to $0$. To do that, we estimate the total variation distance using a result of 
Soon~\cite[Corollary 1.5]{Soon96}, which involves all the covariances $\cov(M_i,M_j)$. 
Soon's estimate, applied to our situation, says that
\begin{equation}  
d_{TV}(M, M') \le C_{k,n} \sum_{j \ne i} |\cov(M_i,M_j)| ,  \label{est:dtv} 
\end{equation}
where 
\[ 
C_{k,n}= \frac{1-p^{kn}-(1-p)^{kn}}{knp(1-p)} 
\quad \text{and} \quad p=\frac{k-1}{kn-1}.
\]
It is straightforward to show that $\lim_{n \ra \infty} C_{n,k}=1/(k-1)$.
 The covariance of $M_i$ and $M_j$ is given by  
\[  
\cov(M_i,M_j)=\E(M_i M_j)-\E(M_i)\E(M_j).
\] 
We have $\E(M_i)\E(M_j)=(k-1)^2/(kn-1)^2$. 
The calculation of $\cov(M_i M_j)$ is carried out through the following three lemmas.

\begin{lemma}
If $\abs{i-j}=1$, then 
\[ 
\E(M_i M_j) = \frac{(k-1)(k-2)}{(kn-1)(kn-2)}.
\]
\end{lemma} 
\begin{proof}
We may assume $j=i+1$. Then $\E(M_i M_j)$ is the probability that cards in location 
$i$, $i+1$, and $i+2$ are the same rank. The first card can be anything. 
The probability that the next card has the same rank is $(k-1)/(kn-1)$, and then 
the probability that the third card has the same rank is $(k-2)/(kn-2)$. 
 \end{proof}
 
\begin{lemma}
If $\abs{i-j} > 1$, then
\[ 
\E(M_i M_j) = \frac{(k-1)(k-2)(k-3)+(k-1)^2(kn-k)}{(kn-1)(kn-2)(kn-3)}.
\]
\end{lemma}
\begin{proof}
In this case, the locations of the matches are separated. 
There are two possible configurations, namely
\[  
\ldots xx \ldots xx\ldots \quad \mbox{or} \quad  \ldots xx \ldots yy \ldots 
\] 
In the first configuration, there are four cards of the same rank; 
in the second the matches involve cards of different ranks. 
The probability of the first is 
\[ 
\frac{(k-1)(k-2)(k-3)}{(kn-1)(kn-2)(kn-3)}.
\]
Note that this configuration cannot occur for $k \le 3$, and the formula correctly 
gives $0$ in that case. For the second configuration, the probability is
\[ 
\frac{(k-1)^2(kn-k)}{(kn-1)(kn-2)(kn-3)} 
\] 
because $(k-1)/(kn-1)$ is the probability of the $xx$ match. 
Then $y$ must be one of $kn-k$ cards of the remaining $kn-2$ and the second 
$y$ must be one of $k-1$ cards of the remaining $kn-3$. The sum of the probabilities 
for the two configurations is $\E(M_i M_j)$.
\end{proof}

\begin{lemma}
\[ 
\cov(M_i,M_j) = 
\begin{cases} 
\displaystyle \frac{(k-1)(k-kn)}{(kn-1)^2(kn-2)}  &\text{if $|i-j|=1$} ,\\[2em]
\displaystyle \frac{2(k-1)(kn-k)}{(kn-1)^2(kn-2)(kn-3)} &\text{if $|i-j| >1$} .
\end{cases}
\]
\end{lemma}
\begin{proof}
Routine algebra left to the reader.
\end{proof}
\noindent
Notice that the covariance is negative when $|i-j|=1$ and positive when $|i-j|>1$. 

\begin{theorem}
The total variation distance between the distributions of 
$M$ and $M'$ goes to $0$ as $n \ra \infty$. 
\end{theorem}

\begin{proof}
In the estimate given in equation~\eqref{est:dtv}, 
the constant $C_{k,n}$ has a limit and so it is bounded. 
We need to show that 
$\sum_{i \ne j} |\cov(M_i,M_j)| \ra 0$
as $n \ra \infty$. Break the sum into two pieces. 
There are $2(kn-2)$ terms with $|i-j|=1$ and the covariances are negative.  From the previous lemma we get
\begin{align*}
\sum_{|i-j|=1}| \cov(M_i,M_j)| 
&=- 2(kn-2) \frac{(k-1)(k-kn)}{(kn-1)^2(kn-2)} \\
&= \frac{2(k-1)(kn-k)}{(kn-1)^2}   .  
\end{align*}  
There are $(kn-2)(kn-3)$ terms with $|i-j|>1$, and the covariances are positive. Again from the lemma we get
\begin{align*}   
\sum_{|i-j|>1} |\cov(M_i,M_j) | 
&=(kn-2)(kn-3) \frac{(k-1)(2kn-2k)}{(kn-1)^2(kn-2)(kn-3)}\\
&= \frac{2(k-1)(kn-k)}{(kn-1)^2}.
\end{align*}

It follows that
\[ 
\lim_{n \ra \infty}\sum_{i \ne j} |\cov(M_i,M_j)|=
\lim_{n \ra \infty}\frac{4(k-1)(kn-k)}{(kn-1)^2}=0.
\]
\end{proof}
\begin{corollary}
As $n \ra \infty$ the distribution of $M$ converges to a Poisson distribution 
with parameter $\lambda=k-1$. That is, for $r \ge 0$,
\[ 
\lim_{n \ra \infty} \pr(M=r) = e^{k-1}\frac{(k-1)^r}{r!}. 
\]
\end{corollary}
\begin{proof}
From the theorem, we know that for  
$r \ge 0$, we have 
\[ 
\lim_{n \ra \infty} |\pr(M=r)-\pr(M'=r) |=0. 
\]
Thus, $\pr(M=r)$ and $\pr(M'=r)$ have the same limit if the limit exists. 
In fact, the limit does exist because the distribution of $M'$ converges to the 
Poisson distribution with parameter $k-1$. This is a classical result which says 
that for a sequence $p_N>0$ such that $Np_N \ra \lambda$, the binomial 
distribution with parameters $N$ and $p_N$ converges to the Poisson distribution 
with parameter $\lambda$. This means that for each $r \ge 0$,
\[  
\lim_{N \ra \infty} \binom{N}{r}p_N^r (1-p_N)^{N-r} = 
e^{-\lambda}\frac{\lambda^r}{r!}.
\]
For $M'$ we have $N=kn-1$, $p_N=(k-1)/(kn-1)$, and $\lambda =k-1$.
\end{proof}

We have noted the curious fact that $M$ and $M'$ have the same variance. 
A more intuitive explanation would be of interest, but our proof is simply the 
computation that depends on the fact that
$\sum_{|i-j|=1}\cov(M_i,M_j) = - \sum_{|i-j|>1}\cov(M_i,M_j)$.

\begin{corollary} \label{same-variance}
\[ 
\var(M) = \var(M') =\frac{k(k-1)(n-1)}{kn-1}. 
\]
\end{corollary}
\begin{proof}
\begin{align*}
\var(M) &= \var(\sum_i M_i) 
= \sum_i \var(M_i) + \sum_{i \ne j} \cov(M_i,M_j) \\
&= \sum_i \var(M_i) + \sum_{|i-j|=1} 
\cov(M_i,M_j) + \sum_{|i-j|>1} \cov(M_i,M_j) = \sum_i \var(M_i).                
\end{align*}
Since $M'$ is the sum of \emph{independent} Bernoulli random variables with 
the same distribution as the $M_i$, we also have $\var(M')=\sum_i \var(M_i)$.
The $M_i$ are Bernoulli random variables with parameter $p=(k-1)/(kn-1)$ 
and variance $p(1-p)$. Thus, 
\begin{align*}
\sum_i \var(M_i) &= \sum_i p(1-p) = (kn-n) p(1-p)\\ 
&= \frac{k(k-1)(n-1)}{kn-1} .
\end{align*}
\end{proof}

\section{Two Ranks ($n=2$)}

For $n$ even moderately large, the binomial approximation is quite good, 
but it is not so good for small values of $n$, especially for $n=2$. In that case, 
however, counting permutations according to the number of matches is quite 
tractable for general $k$, and it does not use the general formula developed in 
section 5, which becomes unwieldy as $k$ increases.

For $n=2$ and $k \ge 1$, the number of permutations is $\binom{2k}{k}$. 
If a permutation begins with $1$, then it is made up of a string of $1$'s, followed 
by a string $2$'s, then a string of $1$'s, and so on. 

Suppose $s$ is even, say $s=2t$.  Then the permutation is made up of $t$ strings 
of $1$'s and $t$ strings of $2$'s that are interlaced, and so it is completely 
determined by two increasing sequences
\[   
1 \le u_1 < u_2 < \cdots < u_{t-1} \le k-1, 
\quad 1 \le v_1 < v_2 < \cdots < v_{t-1} \le k -1
\]
where $u_i$ is the total number of $1$'s that are in the first $i$ strings of $1$'s 
and $v_i$ is the total number of $2$'s in the first $i$ strings of $2$'s. For example, 
if $k=6$ and the permutation is $112111222122$, then $t=3$, $u_1=2, u_2=5$ 
and $v_1=1,v_2=4$.  The number of permutations is $\binom{k-1}{t-1}^2$.

Suppose $s$ is odd, say $s=2t+1$. Then there are $t+1$ strings of $1$'s and $t$ 
strings of $2$'s so that the permutation is determined by sequences
\[  
1 \le u_1 < u_2 < \cdots < u_{t} \le k-1, 
\quad 1 \le v_1 < v_2 < \cdots < v_{t-1} \le k-1.
\]
Therefore, the number of these permutations is $\binom{k-1}{t}\binom{k-1}{t-1}$.

\begin{theorem}  \label{formula:n=2}
If $n=2$, $k \ge 1$, and $0 \le r \le 2k-2$, then $\alpha_r$, the number of 
permutations with exactly $r$ matches, is given by
\[ 
\alpha_r = 
\begin{cases}
2 \displaystyle \binom{k-1}{k-\ell - 1}^2 &\text{if $r=2\ell$}, \\[2em]
2 \displaystyle \binom{k-1}{k-\ell-1} \binom{k-1}{k-\ell-2} &\text{if $r=2\ell +1$}.
\end{cases}
\]
\end{theorem}

\begin{proof}
Above we have counted the permutations beginning with $1$ and made up of $s$ strings. 
In a deck of $2k$ cards there are $2k-1$ locations where there could be a match, 
but each boundary between two strings reduces the number of matches by one. 
There are $s-1$ boundaries between strings. Therefore, $r=2k-1-(s-1)=2k-s$. 
If $r$ is even, say $r=2\ell$, then $s$ is even since $r=2k-s$. 
Letting $s=2t$ we see that $t=k-\ell-1$. Thus, the number of permutations beginning 
with $1$ and having $r$ matches is
\[ 
\binom{k-1}{t-1}^2 = \binom{k-1}{k-\ell - 1}^2.
\]
If $r$ is odd, say $r=2\ell+1$, then $s$ is odd, say $s=2t+1$, and $t=k-\ell-1$. 
The number of permutations beginning with $1$ and having $r$ matches is
\[
\binom{k-1}{t}\binom{k-1}{t-1}=\binom{k-1}{k-\ell-1} \binom{k-1}{k-\ell-2}. 
\]
Now we double these numbers because there are just as 
many permutations that start with $2$.
\end{proof}

You may have noticed in Example~\ref{exm:k=4,n=2} that the generating function 
$A(x)$ is palindromic. That is, $\alpha_r = \alpha_{2k-2-r}$. That can be proved in 
general for any $k$ as long as $n=2$ either from the formula in 
Proposition~\ref{formula:n=2} or from the correspondence between permutations 
and the sequences $(u_i)$ and $(v_i)$. The second approach replaces each sequence 
by its complement and has the virtue of giving a bijection between permutations with 
$r$ matches and those with $2k-2-r$ matches. The bijection is not at all obvious. 
In the case that $r=2k-2-r$, the bijection is not the identity map. 
The palindromic property does not continue for $n \ge 3$. In particular, 
although $\alpha_{kn-n}=n!$, you can see that $\alpha_0$, the number of 
permutations with no matches, is larger simply by constructing enough of them.

The numbers $\alpha_r$ for $k\ge 1$ and $0\le r \le 2k-2$ 
appear as entry A152659 in the OEIS, but there they are described as the number 
of lattice paths from $(0,0)$ to $(n,n)$ with steps $(1,0)$ and $(0,1)$ 
and having $k$ turns. To see the correspondence between permutations and paths, 
you must replace $n$ with $k$ and replace $k$ with $2k-r-1$.

\quad

\section*{Appendix: Mathematica and Sage}
The Mathematica code below samples $100$,$000$ random deals and counts the 
number of matches in each deal. Then it finds the average and plots a histogram. 
The deck is the multiset with four copies of each number from $1$ to $13$.

\begin{verbatim}
deck = Flatten[Table[{i, i, i, i}, {i, 1, 13}]];
matches = {};
Do[s = RandomSample[deck, 52];
   AppendTo[matches, Count[Differences[s], 0]], {100000}];
Mean[matches]
Histogram[matches, {1}, "Probability", 
          Ticks->{Range[0,13],Automatic}]
\end{verbatim}

\quad

Mathematica code for $\beta_m$ and $\alpha_r$ with $n$ ranks and four suits. 
\begin{verbatim}
Needs["Combinatorica`"]
beta[n_, m_] := 
  Module[{z=Select[Compositions[n,5], 
                      #[[2]]+2#[[3]]+2#[[4]]+3#[[5]]==m&]},
   Sum[Multinomial[z[[i,1]],z[[i,2]],z[[i,3]],z[[i,4]],z[[i,5]]] 
      ({4,3,2,2,1}.z[[i]])!/(24^z[[i,1]]*2^(z[[i,2]]+z[[i,4]])), 
      {i,1,Length[z]}]]              
alpha[n_, r_] :=  Sum[(-1)^m beta[n, m], {m, r, 3n}] 
\end{verbatim}  

\quad 

Sage code for the random samples and histogram.

\begin{verbatim}
def differences(x):
    return [x[i+1]-x[i] for i in range(len(x)-1)]
deck = 4*[1..13]    
matches = []
for i in range(100000):
    x = Permutations(deck).random_element()
    matches.append(differences(x).count(0))   
import numpy
import matplotlib.pyplot as plt
plt.figure(figsize=(7.5, 5))
plt.hist(matches,[i-.5 for i in range(13)], density=True,\
            edgecolor="white")
plt.xticks(range(14))
plt.show()
\end{verbatim}
 
 \quad 
 
Sage code for $\beta_m$ and $\alpha_r$ with $n$ ranks and four suits.
\begin{verbatim}
def beta(n,m):
    cmp=Compositions(n,length=5,min_part=0)
    z=[c for c in cmp if c[1]+2*c[2]+2*c[3]+3*c[4]==m]
    return sum(multinomial(list(x))
               *factorial(4*x[0]+3*x[1]+2*x[2]+2*x[3]+x[4]) 
               /((24)^x[0]*2^x[1]*2^x[3]) for x in z)
def alpha(n,r):
    return sum((-1)^(m-r)*binomial(m,r)*beta(n,m)\
      for m in range(r,3*n+1))
\end{verbatim}

\;

\noindent \textbf{Acknowledgments} \quad Thanks to David Farmer who asked about the probability of no matches, telling me that as a kid when he had nothing to do he would deal out the deck and---as far as he could remember---never fail to get at least one match.



\begin{thebibliography}{99}
\bibitem{Bogart&Doyle86}K. P. Bogart and P. G. Doyle, Non-sexist solution of the m\'enage problem, \emph{Amer. Math. Monthly} {\bf 93} (1986) 514--518. \url{doi.org/10.2307/2323022}

\bibitem{Eriksson&Martin2017}H. Eriksson and A. Martin, Enumeration of Carlitz permutations, \emph{arXiv} (2017). \\
\url{doi.org/10.48550/arXiv.1702.04177}

\bibitem{Soon96}S. Y. T. Soon, Binomial approximation for dependent indicators, \emph{Stat. Sinica}. {\bf 6 }(1996) 703--714. \\
\url{https://www.jstor.org/stable/24305618}

\bibitem{Wilf94} H. S. Wilf,
\newblock \emph{generatingfunctionology.} 2nd ed. Academic Press, San Diego, 1994. \\
\url{https://www.math.upenn.edu/~wilf/gfologyLinked2.pdf}

\end{thebibliography}
\end{document}